\documentclass[12pt,a4paper,reqno]{amsart}
\usepackage{amsmath, amsthm, amsfonts, amssymb}
\usepackage{enumerate}
\usepackage{hyperref}
\usepackage[dvips]{graphicx}

\numberwithin{equation}{section}
\theoremstyle{plain}
\newtheorem{theorem}{Theorem}[section]
\newtheorem{lemma}[theorem]{Lemma}

\newtheorem{corollary}[theorem]{Corollary}
\newtheorem{proposition}[theorem]{Proposition}

\theoremstyle{definition}
\newtheorem{defn}[theorem]{Definition}
\newtheorem{remark}[theorem]{Remark}

\def\Z{\mathbb{Z}}
\def\R{\mathbb{R}}
\def\T{\mathbb{T}}
\def\C{\mathbb{C}}
\def\N{\mathbb{N}}

\def\Ghat{\widehat{G}}
\def\Hhat{\widehat{H}}
\def\fhat{\widehat{f}}
\def\ghat{\widehat{g}}
\newcommand{\eulerian}[2]{\genfrac{\langle}{\rangle}{0pt}{}{#1}{#2}}
\newcommand{\ud}{\,\mathrm{d}}

\newcommand{\Zmod}[1]{\Z_{#1}} 
\providecommand{\abs}[1]{\lvert#1\rvert}
\providecommand{\norm}[1]{\lVert#1\rVert}

\providecommand{\Tsol}{S}
\providecommand{\md}{d}

\title[Maximal densities of sets avoiding linear equations modulo a prime]{On the asymptotic maximal density of a set avoiding solutions to linear equations modulo a prime}

\author{Pablo Candela}
\address{Centre for Mathematical Sciences\\
	Wilberforce Road\\
	Cambridge CB3 0WB\\
	United Kingdom}
\email{pc308@cam.ac.uk}

\author{Olof Sisask}
\address{School of Mathematical Sciences\\
	Queen Mary, University of London\\
	Mile End Road\\
	London E1 4NS\\
	United Kingdom}
\email{O.Sisask@qmul.ac.uk}

\thanks{Both authors are EPSRC postdoctoral fellows and gratefully acknowledge the support of the EPSRC}


\begin{document}

\begin{abstract}
Given a finite family $\mathcal{F}$ of linear forms with integer coefficients, and a compact abelian group
$G$, an $\mathcal{F}$-free set in $G$ is a measurable set which does not contain solutions to any equation $L(x)=0$ for $L$ in $\mathcal{F}$. We denote by $d_\mathcal{F}(G)$ the supremum of $\mu(A)$ over $\mathcal{F}$-free sets $A\subset G$, where $\mu$ is the normalized Haar measure on $G$. Our main result is that, for any such collection $\mathcal{F}$ of forms in at least three variables, the sequence $d_\mathcal{F}(\Zmod{p})$ converges to $d_\mathcal{F}(\R/\Z)$ as $p\to\infty$ over primes. This answers an analogue for $\Zmod{p}$ of a question that Ruzsa raised about sets of integers.
\end{abstract}

\maketitle

\section{Introduction}
Much work in arithmetic combinatorics concerns the maximal density that a subset of a finite abelian group can have if it does not contain a non-trivial solution to a given linear equation. Examples include the study of sum-free sets, where the equation to be avoided is $x_1+x_2-x_3=0$, and the improvement of bounds for Roth's theorem, which concerns the equation $x_1-2x_2+x_3=0$.

A natural question about these maximal densities is whether they exhibit some particular asymptotic behaviour as the groups get larger in a family of groups, and a typical family to consider in this context is that of the groups $\Zmod{p}$ of residues modulo a prime. One may thus ask whether the maximal density of a subset of $\Zmod{p}$ avoiding solutions to some linear equation converges as $p\to \infty$. The analogous question for subsets of the integers $\{1, \ldots, N\}$ was raised by Ruzsa \cite[Problem 2.3]{ruzsa:linear-equationsII}, who conjectured that the corresponding limit exists for any linear equation. The main result in this paper implies that for the groups $\Zmod{p}$ the maximal density does indeed converge. The result itself is more general, and to state it precisely we shall use the following notation.
\begin{defn}\label{L-free}
Let $L(x) = c_1 x_1 + \cdots + c_t x_t$ be a linear form with non-zero integer coefficients. We say that a subset $A$ of an abelian group $G$ is \emph{$L$-free} if there is no $t$-tuple $x=(x_1,\ldots,x_t) \in A^t$ such that $L(x) = 0$. For a family $\mathcal{F}$ of linear forms we say that $A$ is \emph{$\mathcal{F}$-free} if $A$ is $L$-free for every $L \in \mathcal{F}$.
\end{defn}
For a given family $\mathcal{F}$ of linear forms, we define
\[ \md_\mathcal{F}(\Zmod{p}) = \max\{ |A|/p : \text{$A \subset \Zmod{p}$ and $A$ is $\mathcal{F}$-free} \}. \]
Our main result implies firstly that for any finite family $\mathcal{F}$ of forms in at least 3 variables, the maximal density $\md_\mathcal{F}(\Zmod{p})$ converges as $p \to \infty$ over primes. This is analogous to a result of Croot \cite{croot} establishing the convergence of the minimal normalised-count of three-term arithmetic progressions in subsets of $\Z_p$ of some fixed density, 
and indeed a variant of Croot's method will be an essential part of our argument. Before stating the result in full, let us note that convergence can fail in our result if we do not restrict to prime moduli, as was also the case in \cite{croot}. Indeed, sum-free sets in $\Zmod{p}$ are easily shown to have maximal density converging to $1/3$ as $p \to \infty$ over primes (using the Cauchy-Davenport inequality \cite{T-V}), but in $\Zmod{2p}$ the odd residues form a sum-free set of density $1/2$.

The convergence of $\md_\mathcal{F}(\Zmod{p})$ leads to the problem of determining the limit. To address this it is potentially helpful to have a fixed group on which the limit can be analyzed. We show that the circle is one possible such group.
\begin{defn}
For a family $\mathcal{F}$ of linear forms and a compact abelian group $G$ with normalized Haar measure $\mu$, we define
\[\md_\mathcal{F}(G) = \sup \{ \mu(A) : \text{$A \subset G$ is measurable and $\mathcal{F}$-free}\}.\]
\end{defn}

We can now state our main result.

\begin{theorem}\label{mainres}
Let $\mathcal{F}$ be a finite family of linear forms, each in at least three variables.
Then $\md_{\mathcal{F}}(\Zmod{p})\to \md_{\mathcal{F}}(\T)\textrm{ as }p \to \infty\textrm{ over primes.}$
\end{theorem}

The methods in this paper do not allow to extend this result to families containing forms in two variables, essentially because such forms do not
fall under the purview of Fourier analysis; in the language of \cite{GTlin}, equations in two variables are of unbounded complexity, while equations in at least three variables are of complexity 1. (Note however that for a single non-trivial form $L$ in two variables it is easy to show that $d_L(\Zmod{p})\to 1/2$.)

It is an important fact that if $L$ is a translation-invariant form, that is a form $c_1 x_1 + \cdots + c_t x_t$ with $c_1 + \cdots + c_t = 0$, then $\md_L(\Zmod{p}) \to 0$ as $p \to \infty$, even if we allow an $L$-free set
to contain certain trivial solutions (constant solutions $(x,\ldots,x)$, for instance); this follows from Roth's method, as observed in a very similar context in \cite[Theorem 1.3]{ruzsa:linear-equationsI}. Therefore the limit in Theorem \ref{mainres} is also $0$ for any family $\mathcal{F}$ containing a translation-invariant form. On the other hand, if every member of $\mathcal{F}$ is not translation-invariant (to be concise let us say \emph{non-invariant}), then the limit will be positive.
\begin{proposition}\label{positivity}
Let $\mathcal{F}$ be a finite family of non-invariant linear forms, each in at least three variables. Then
$\md_{\mathcal{F}}(\T) = \lim_{p \to \infty} \md_{\mathcal{F}}(\Zmod{p}) > 0$.
\end{proposition}

This situation is thus analogous to that of the maximal edge-densities of graphs with forbidden subgraphs: if one of the forbidden subgraphs is bipartite then the maximal edge-density of a graph on $n$ vertices that does not contain the forbidden subgraphs tends to $0$ as $n \to \infty$; otherwise it tends to a positive limit determined by the chromatic numbers of the forbidden subgraphs.

It would be very interesting to have a simple asymptotic formula for $\md_\mathcal{F}(\Zmod{p})$ in terms of the coefficients of the linear forms, if indeed such a formula exists. While Theorem \ref{mainres} does not give such a formula, it does show that it suffices to analyze families of linear equations on $\T$ in order to find one.

The proof of Theorem \ref{mainres} and the remainder of this paper are laid out as follows. In Section \ref{TL_props} we gather tools that will enable us to use Fourier analysis to tackle the problem. In Section \ref{Transference} we review and extend a result from \cite{OlofTh}, itself based on work of Croot \cite{croot}, that allows us to transfer a solution-free set in $\Zmod{p}$ of density $\alpha$ to an almost-solution-free function on $\T$ with average $\alpha$. Section \ref{functions-to-sets} then shows how one may use this function to produce an almost-solution-free set of density close to $\alpha$; we then use a removal lemma on $\T$, proved in Section \ref{Removal}, to obtain a truly solution-free subset of $\T$ of density close to $\alpha$.
In Section \ref{Final} we combine these results to prove Theorem \ref{mainres}, and we also prove Proposition \ref{positivity}.

The transference result in Section \ref{Transference} concerns general compact abelian groups and we shall therefore work in
this general setting until the end of that section. From Section \ref{functions-to-sets} onwards we shall focus on the groups $\Zmod{p}$ and $\T$.

\section{Properties of $\Tsol_L$}\label{TL_props}
One of the main tools used in this paper is a multilinear operator associated with a given linear form $L$. This operator, which we denote $\Tsol_L$, is well known in arithmetic combinatorics in the setting of finite groups. In this section we define this operator and establish some of its main properties.

In a finite abelian group $G$, a set $A$ is $L$-free in the sense of Definition \ref{L-free} if and only if $|A^t\cap \ker L|=0$,
where $t$ is the number of variables in $L$ and $\ker L$ denotes the subgroup $\{x\in G^t:L(x)=0\}$ of the direct product $G^t$.
The quantity $|A^t\cap \ker L|/|\ker L|$ is the normalized count of solutions to $L(x)=0$ inside $A$. This notion has a natural generalization to any compact abelian group $G$, using the normalized Haar measure on $\ker L$.

\begin{defn}\label{TLs}
Let $G$ be a compact abelian group, let $L$ be a linear form in $t$ variables, and let $\mu_L$ denote the Haar measure on the closed subgroup $\ker L$ of $G^t$ satisfying $\mu_L(\ker L)=1$. Then for any $t$ measurable sets $A_1,A_2,\ldots,A_t$ we define the \emph{solution measure}
$\Tsol_L(A_1,\ldots,A_t) = \mu_L\big( (A_1\times \cdots\times A_t) \cap \ker L\big)$.
When $A_1=A_2=\cdots=A_t=A$ we write more concisely $\Tsol_L(A)$.
\end{defn}

The operator $\Tsol_L$ is the result of extending this definition from sets to functions.

\begin{defn}\label{TLf}
Let $G$ be a compact abelian group, let $L$ be a linear form in $t$ variables, and let $f_1,f_2,\ldots,f_t : G\rightarrow \C$ be measurable functions. We then define
\begin{equation}\label{TLfeq}
\Tsol_L(f_1,\ldots,f_t) = \int_{\ker L} f_1\otimes f_2\otimes \cdots \otimes f_t(x) \ud\mu_L(x),
\end{equation}
where $f_1\otimes f_2\otimes \cdots \otimes f_t(x) := f_1(x_1)f_2(x_2)\cdots f_t(x_t)$ for all $x=(x_1,\ldots,x_t) \in \ker L \subset G^t$. When $f_1=f_2=\cdots=f_t=f$ we write more concisely $\Tsol_L(f)$.
\end{defn}

Clearly Definitions \ref{TLs} and \ref{TLf} agree on measurable sets $A_1,\ldots,A_t$, that is we have $\Tsol_L(A_1,\ldots,A_t)=\Tsol_L(1_{A_1},\ldots,1_{A_t})$, where $1_X$ denotes the indicator function of a set $X$.

Let us recall the following standard fact, which will be used throughout the paper.
\begin{lemma}\label{cts-surj-hom}
Let $G$ and $H$ be compact abelian groups and let $\phi : G \to H$ be a surjective continuous homomorphism. Then $\phi$ preserves the normalized Haar measures, that is $\mu_H = \mu_G \circ \phi^{-1}$.
\end{lemma}
We shall also use the Gowers $U^2$ norm.
\begin{defn}
Let $\mathcal{L}_2(G)$ denote the Hilbert space of square-integrable complex-valued functions on $G$. 
The $U^2$ norm $\norm{\cdot}_{U^2(G)}$ can be defined on $\mathcal{L}_2(G)$ by the formula
\begin{equation}\label{U2}
\norm{f}_{U^2(G)}^4=\int_{G^4} f(x+y)\,\overline{f(x+y')}\,\overline{f(x'+y)}\, f(x'+y') \ud x\ud x'\ud y\ud y'.
\end{equation}
When the group $G$ in question is clear we write more concisely $\norm{f}_{U^2}$.
\end{defn}
By a simple application of the Cauchy-Schwarz inequality to the right hand side of \eqref{U2}, we obtain the bound $\norm{f}_{U^2}\leq\norm{f}_{\mathcal{L}_2}$.

For an integer $n$ we let $n\cdot X$ denote the image of $X\subset G$ under the continuous map $x\mapsto n x$. We refer to this map as \emph{dilation} by $n$ or $n$-\emph{dilation}.

The following result is well known in the setting of finite groups and will be used below.
\begin{theorem}\label{U2control}
Let $L$ be a linear form in $t\geq 3$ variables, and let $G$ be a compact abelian group such that each coefficient of $L$ gives a surjective dilation on $G$. Then for any $f_1,\ldots,f_t\in \mathcal{L}_2(G)$ we have
\begin{equation}\label{U2C}
|\Tsol_L(f_1,\ldots,f_t)|\leq \min_{i\in [t]}\; \norm{f_i}_{U^2} \prod_{j\neq i} \norm{f_j}_{\mathcal{L}_2}.
\end{equation}
\end{theorem}
Here $[t]$ denotes the set of integers $\{1,2,...,t\}$.
\begin{proof}
Writing $L(x) = c_1 x_1 + \cdots + c_t x_t$, by assumption we have $c_i \cdot G = G$ for each $i\in [t]$. We shall prove the upper bound in 
\eqref{U2C} just for $i=t$, which is sufficient by symmetry.

By the assumption on $c_j$-dilations, the map $G^{t-1}\to \ker L$, \[(y_1,...,y_{t-1})\mapsto (c_ty_1,c_ty_2,\ldots,c_ty_{t-1},-c_1y_1-c_2y_2-\cdots-c_{t-1}y_{t-1})\]
is surjective, so by Lemma \ref{cts-surj-hom} we have
\begin{equation}\label{param}
\Tsol_L(f_1,\ldots,f_t)=\int_{G^{t-1}} f_1(c_ty_1)\cdots f_{t-1}(c_ty_{t-1})f_t(-c_1y_1-\cdots-c_{t-1}y_{t-1})\ud y_1\cdots \ud y_{t-1}.
\end{equation}
We can use this expression to prove the desired upper bound by an argument which is standard in the setting of finite abelian groups, consisting
in two applications of the Cauchy-Schwarz inequality. We include the details in the present setting for completeness.
First we apply Fubini's theorem and the Cauchy-Schwarz inequality over variables $y_2,...,y_{t-1}$ to obtain
\begin{align*}
|\Tsol_L(f_1,\ldots,f_t)|^2 &\leq \int_{G^{t-2}} |f_2(c_t y_2) \cdots f_{t-1}(c_t y_{t-1})|^2 \ud y_2\cdots \ud y_{t-1} \\
&\quad \cdot \int_{G^{t-2}}\left|\int_G f_1(c_ty_1)f_t(-c_1y_1-\cdots-c_{t-1}y_{t-1}) \ud y_1 \right|^2\ud y_2\cdots \ud y_{t-1}.
\end{align*}
Applying Lemma \ref{cts-surj-hom} to $c_t$-dilation and $(y_2,\ldots,y_{t-1})\mapsto x=-c_2y_2-c_3y_3-\cdots-c_{t-1}y_{t-1}$, the right hand side above is found to equal
\[\norm{f_2}_{\mathcal{L}_2}^2\cdots \norm{f_{t-1}}_{\mathcal{L}_2}^2 \int_G\left|\int_G f_1(c_ty)f_t(x-c_1y) \ud y \right|^2\ud x.\]
By Fubini's theorem the integral here equals
\[\int_{G^2} f_1(c_ty)\overline{f_1(c_ty')} \left(\int_G f_t(x-c_1y)\overline{f_t(x-c_1y')} \ud x\right)\ud y\ud y'.\]
Applying Cauchy-Schwarz over $(y,y')$, Lemma \ref{cts-surj-hom} for dilations by $c_1$ and by $c_t$, and Fubini's theorem again, we find that this integral is at most $\norm{f_1}_{\mathcal{L}_2}^2$ times
\[
\left(\int_{G^4} f_t(x+y)\,\overline{f_t(x+y')}\, \overline{f_t(x'+y)}\, f_t(x'+y') \ud x\ud x'\ud y\ud y'\right)^{1/2}=\norm{f_t}_{U^2}^2.\qedhere
\]
\end{proof}
Note that, since the $\mathcal{L}_2$ norm dominates the $U^2$ norm, \eqref{U2C} implies immediately
\begin{equation}\label{L2control}
|\Tsol_L(f_1,\ldots,f_t)|\leq \prod_i \norm{f_i}_{\mathcal{L}_2}.
\end{equation}

We now turn to Fourier-analytic aspects of $\Tsol_L$. Let us first settle on some notation. For a compact abelian group $G$ and any real $r\geq 1$ we
denote by $\mathcal{L}_r(G)$ the Banach space of complex-valued functions on $G$ with integrable $r$th power. The implicit measure here
is the normalized Haar measure on $G$. The Pontryagin dual $\widehat{G}$ of $G$ is a discrete group, so by $\mathcal{L}_r(\widehat{G})$ we denote 
the analogous Banach space but with the implicit measure being the counting measure on $\widehat{G}$, which assigns value 1 to each singleton.

We shall use the Fourier transform on $\mathcal{L}_2(G)$ (or Plancherel transform). By the compactness of $G$ we have $\mathcal{L}_2(G)\subset\mathcal{L}_1(G)$.
If $f\in \mathcal{L}_1(G)$ then the Fourier transform $\widehat{f}$ is defined for $\gamma\in \widehat{G}$ by
$\widehat{f}(\gamma)=\int_G f(x)\overline{\gamma(x)}\ud\mu_{G}(x)$.
If $\widehat{f}$ is also in $\mathcal{L}_1(\widehat{G})$ then we have the Fourier inversion formula
\begin{equation}\label{inversion}
f(x)=\int_{\widehat G} \widehat{f}(\gamma)\gamma(x) \ud\mu_{\widehat{G}}(\gamma).
\end{equation}
Plancherel's theorem gives us that for any $f\in \mathcal{L}_2(G)$, $\widehat{f}\in\mathcal{L}_2(\widehat{G})$ and $\norm{f}_{\mathcal{L}_2(G)}=\norm{\widehat{f}}_{\mathcal{L}_2(\widehat{G})}$.

The following expression of $\Tsol_L$ in terms of the Fourier transforms $\widehat{f_i}$ will be used in the next section.

\begin{proposition}\label{TLinversion}
Let $L = c_1 x_1 + \cdots + c_t x_t$ be a linear form and let $G$ be a compact abelian group such that
$c_i\cdot G=G$ for every $i$. Then for any $f_1,\ldots,f_t\in\mathcal{L}_2(G)$ we have
\begin{equation}\label{FI}
\Tsol_L(f_1,\ldots,f_t)= \int_{\widehat{G}} \widehat{f_1} (\gamma^{c_1}) \cdots \widehat{f_t} (\gamma^{c_t}) \ud\mu_{\widehat{G}}(\gamma).
\end{equation}
\end{proposition}
\begin{proof}
Let $H=\ker L$. First we prove \eqref{FI} for $f_1,\ldots,f_t$ having Fourier transforms in $\mathcal{L}_1(\widehat{G})$. In this case each $f_i$ is continuous (by \eqref{inversion}), and it follows that the function $F$ defined on $G^t/H$ by $F(x+H)= \int_H f(x+y)\ud\mu_H(y)$ is continuous on $G^t/H$. Note that $\Tsol_L(f_1,\ldots,f_t)=F(0)$. Let us consider the Fourier coefficients $\widehat{F}(\chi)$, $\chi\in \widehat{G^t/H}\cong H^\perp$. On one hand, it follows from $c_i\cdot G=G$ that $\gamma\mapsto \chi=\gamma\circ L$ is an isomorphism $\widehat{G}\to H^{\perp}$. On the other hand, using a standard formula for integration on quotient groups \cite[2.7.3 (2)]{RudLCA} one checks that $
\widehat{F}(\gamma\circ L)=\int_{G^t}f(x)\overline{\gamma\circ L(x)} \ud\mu_{G^t}(x)=\widehat{f_1}(\gamma^{c_1})\cdots \widehat{f_t}(\gamma^{c_t})$.
Therefore, provided $\widehat{F}$ is in $\mathcal{L}_1(H^\perp)$, we can apply Fourier inversion to conclude that
\[
\Tsol_L(f_1,\ldots,f_t)=F(0)=\int_{H^\perp} \widehat{F}(\chi) \ud\mu_{H^\perp}(\chi) =\int_{\widehat{G}} \widehat{f_1}(\gamma^{c_1})\cdots \widehat{f_t}(\gamma^{c_t})\ud\mu_{\widehat{G}}(\gamma).
\]
The function $\gamma\mapsto \widehat{f_1}(\gamma^{c_1})\cdots \widehat{f_t}(\gamma^{c_t})$ is shown to be indeed in $\mathcal{L}_1(\widehat{G})$ using Cauchy-Schwarz, Plancherel's theorem, and the fact that $\gamma\mapsto \gamma^{c_i}$ is injective (i.e. that $c_i\cdot G=G$).

Now let $f_i\in\mathcal{L}_2(G)$, so $\widehat{f_i}\in\mathcal{L}_2(\widehat{G})$ with $\norm{\widehat{f_i}}_{\mathcal{L}_2(\widehat{G})}=\norm{f_i}_{\mathcal{L}_2(G)}$. For each $i$ we have a sequence $\widehat{g}_{i,n}$ in
$\mathcal{L}_1(\widehat{G})$ with $\widehat{g}_{i,n}\to \widehat{f_i}$ in the $\mathcal{L}_2(\widehat{G})$ norm, whence also $g_{i,n}\to f_i$ in $\mathcal{L}_2(G)$. One then uses multilinearity of $\Tsol_L$, \eqref{L2control}, and Cauchy-Schwarz on $\mathcal{L}_2(\widehat{G})$ to show that
\begin{eqnarray*}
\Tsol_L(f_1,\ldots,f_t)&=&\lim_{n\to\infty} \Tsol_L(g_{1,n},g_{2,n},\ldots,g_{t,n})\\
&=&\lim_{n\to\infty}\int_{\widehat{G}} \widehat{g}_{1,n}(\gamma^{c_1})\cdots \widehat{g}_{t,n}(\gamma^{c_t}) \ud\mu_{\widehat{G}}(\gamma)
=\int_{\widehat{G}} \widehat{f}_1(\gamma^{c_1})\cdots \widehat{f}_t(\gamma^{c_t})\ud\mu_{\widehat{G}}(\gamma).
\end{eqnarray*}
\end{proof}
We close this section with the observation that for the proof of Theorem \ref{mainres} we can assume that each $L\in\mathcal{F}$ has coprime coefficients. This is justified by the following lemma.

\begin{lemma}
Let $\mathcal{F}$ be a finite family of linear forms, and let $\mathcal{F}'$ be obtained by multiplying each $L\in\mathcal{F}$ by some non-zero integer $n_L$. Let $G$ be a compact abelian group such that $n_L\cdot G=G$ for every $L$. Then $\md_\mathcal{F}(G)= \md_{\mathcal{F}'}(G)$.
\end{lemma}
\begin{proof}
It is clear that any $nL$-free set is also $L$-free, so $\md_{\mathcal{F}'}(G) \leq \md_\mathcal{F}(G)$. On the other hand, if $A$ is $L$-free then $n^{-1}A$ is $nL$-free for any $n$, for if $x_i \in n^{-1}A$ and $(n c_1) x_1 + \cdots + (n c_t) x_t = 0$ then $(n x_1, \ldots, n x_t) \in A^t$ is a solution. If in addition $n \cdot G = G$ then $\mu(n^{-1}A) = \mu(A)$. These properties imply easily that if $A$ is $\mathcal{F}$-free then $m^{-1} A$ is $\mathcal{F}'$-free, where $m = \prod_{L \in \mathcal{F}} n_L$, and so $\md_\mathcal{F}(G) \leq \md_{\mathcal{F}'}(G)$.
\end{proof}

\section{Transference}\label{Transference}
In proving Theorem \ref{mainres} we shall need to move sets between the groups $\Zmod{p}$ and $\T$. A result essentially allowing us to do so was established in Chapter 4 of \cite{OlofTh} by extending ideas of Croot \cite{croot}, though for simplicity it was assumed there that all linear forms considered had at least one coefficient equal to $1$. In this section we shall review this result and also show how to eliminate the assumption on the coefficients, thus obtaining Proposition \ref{transfer} and, as our main application, Corollary \ref{transfer-cor}.

Central to the results we are about to discuss is the notion of Freiman isomorphism.

\begin{defn} Let $k \geq 2$ and let $A \subset G$, $B \subset H$ be subsets of two abelian groups. We call a function $\varphi : A \to B$ a \emph{Freiman $k$-isomorphism} if it is a bijection and
\[ a_1 + \cdots + a_k = a_{k+1} + \cdots + a_{2k} \Longleftrightarrow \varphi(a_1) + \cdots + \varphi(a_k) = \varphi(a_{k+1}) + \cdots + \varphi(a_{2k}) \]
for all $a_i \in A$.
\end{defn}

Thus Freiman $k$-isomorphisms, or just $k$-isomorphisms for short, preserve additive relations of length at most $k$. The main result of \cite[Chapter 4]{OlofTh} was that if one can find $k$-isomorphisms between small subsets of the duals of $G$ and $H$, then one can model functions on one group by functions on the other, in a particular sense. That sense uses the following notion of \emph{admissibility} of a linear form.

\begin{defn}\label{GenAdm}
Let $G$ be a compact abelian group and let $L = c_1 x_1 + \cdots + c_t x_t$ be a linear form in at least three variables with coprime coefficients.
We say that $L$ is \emph{$G$-admissible} if $c_i \cdot G = G$ for each coefficient $c_i$. Since the $c_i$ are coprime there are integers $n_i$ such that $n_1 c_1 + \cdots + n_t c_t = 1$; we call the minimum value of $\abs{n_1} + \cdots + \abs{n_t}$ over such integers the \emph{multiplier-height} of $L$ and we denote this quantity by $h(L)$. We shall say that $L$ is \emph{$k$-admissible} if
\begin{equation}\label{klb}
k \geq \max\{ h(L), \abs{c_1}, \abs{c_2}, \ldots, \abs{c_t} \}.
\end{equation}
We give the obvious meaning to $(G,k)$-admissibility and, if $H$ is another compact abelian group, to $(G,H,k)$-admissibility.
\end{defn}

One more definition is needed in order to state the transference result.

\begin{defn}
Let $G$ and $H$ be abelian groups. We say that \emph{$H$ can Freiman $(n,k)$-model $G$} if for any subset $A$ of $G$ of size $|A|\leq n$ there exists a Freiman $k$-isomorphism from $A$ to a subset of $H$.
\end{defn}

\begin{proposition}\label{transfer}
Let $\epsilon>0$ and $h$ be a positive integer. Suppose $G$ and $H$ are two compact abelian groups such that,
for some $k \geq k_0(\epsilon)$ and all $n \leq C(\epsilon)^h$, $\widehat{H}$ can Freiman $(n,k)$-model $\widehat{G}$. Let $f : G \to [0,1]$ be a measurable function with $\int_G f = \alpha$. Then there is a continuous function $g : H \to [0,1]$ with $\int_H g = \alpha$ such that
\[ \abs{\Tsol_L(f) - \Tsol_L(g)} \leq t \epsilon \alpha^{t-2} \]
for any $(G,H,k)$-admissible linear form $L$ in $t \geq 3$ variables with multiplier-height at most $h$.
\end{proposition}

The main difference between this proposition and \cite[Proposition 4.2.8]{OlofTh} lies in the parameter $h$: since the latter result only dealt with forms where at least one coefficient was equal to $1$ one could always take $h = 1$. The purpose of the rest of this section is to indicate the proof of \cite[Proposition 4.2.8]{OlofTh} and show how one may adapt it to take the parameter $h$ into account.

\emph{Step 1: regularize $f$.} The proof begins with a modification of a standard procedure: one replaces $f$ by a function $f' : G \to [0,1]$ with Fourier support $R \subset \widehat{G}$ of bounded size, with $R$ containing the identity $1_{\widehat{G}}$ and with $R = R^{-1}$, so that
\[ f'(x) = \sum_{\gamma \in R} \widehat{f'}(\gamma)\gamma(x) \]
for all $x \in G$ and $\abs{R} \leq C(\epsilon)$. One can do this in such a way that $\int_G f' = \int_G f$ and $\abs{\Tsol_L(f) - \Tsol_L(f')} \leq t \epsilon \alpha^{t-2}/2$ for any $G$-admissible linear form $L$ in $t$ variables.

\emph{Step 2: transfer to $H$.} By the assumption that $\widehat{H}$ can $(n,k)$-model $\widehat{G}$ one can find a $k$-isomorphism $\varphi : R \to R' \subset \widehat{H}$. One then defines $g : H \to \R$ by
\[ g(x) := \sum_{\gamma \in R} \widehat{f'}(\gamma) \varphi(\gamma)(x). \]
The key properties of $g$ used in \cite{OlofTh} were that $\int_H g = \int_G f$ and that $\Tsol_L(g) = \Tsol_L(f')$ for any $(G,H,k)$-admissible form $L$ that has at least one coefficient equal to $1$; we shall therefore need to modify this part of the argument.

\emph{Step 3: control the range of $g$.} The function $g$ produced in the previous step does not a priori take values in $[0,1]$, as we need it to do. However, one can then produce a function $g' : H \to [0,1]$ such that $\int_H g' = \int_H g$ and $\abs{\Tsol_L(g) - \Tsol_L(g')} \leq t \epsilon \alpha^{t-2}/2$ for any $H$-admissible linear form $L$ (there are several ways to do this; see \cite{OlofTh}). This step completes the proof.

\begin{proof}[Proof of Proposition \ref{transfer}]
We shall only need to modify Step 2 above. We are thus given a function $f : G \to \R$ and a symmetric set $R \subset \widehat{G}$ containing the identity $1_{\Ghat}$ such that
\[ f(x) = \sum_{\gamma \in R} \widehat{f}(\gamma)\gamma(x) \]
for all $x \in G$, and we have that $R$ has bounded size: $\abs{R} \leq C(\epsilon)$. Let $Q \subset \widehat{G}$ be the set $Q := R^h = R \cdot R \cdots R $ where $h$ is the height-parameter supplied to Proposition \ref{transfer}. Certainly we have $\abs{Q} \leq C(\epsilon)^h$, and so the modelling hypothesis guarantees the existence of a Freiman $k$-isomorphism $\varphi : Q \to Q' \subset \widehat{H}$. By translating we may assume that $\varphi(1_{\Ghat}) = 1_{\Hhat}$. Note that this means that
\[ \gamma_1^{r_1} \cdots \gamma_n^{r_n} = \gamma_1^{s_1} \cdots \gamma_n^{s_n} \Longleftrightarrow \varphi(\gamma_1)^{r_1} \cdots \varphi(\gamma_n)^{r_n} = \varphi(\gamma_1)^{s_1} \cdots \varphi(\gamma_n)^{s_n} \]
whenever $\gamma_i \in Q$ and the $r_i, s_i$ are non-negative integers with $\sum r_i \leq k$ and $\sum s_i \leq k$. Note also that $\varphi(\gamma^{-1}) = \varphi(\gamma)^{-1}$ for all $\gamma \in Q$ since $\varphi(1_{\Ghat}) = 1_{\Hhat}$. 
We now establish the following lemma, which replaces \cite[Lemma 4.6.3]{OlofTh}.

\begin{lemma}
Define $g : H \to \R$ by setting
\begin{equation} g(x) := \sum_{\gamma \in R} \widehat{f}(\gamma) \varphi(\gamma)(x) \label{g_formula} \end{equation}
for each $x \in H$. Then $\int_H g = \int_G f$, and if $L$ is a $(G,H,k)$-admissible linear form of multiplier-height at most $h$ then $\Tsol_L(g) = \Tsol_L(f)$.
\end{lemma}
\begin{proof}
The properties follow from the fact that \eqref{g_formula} is the Fourier expansion of $g$ and from the Fourier-inversion of $\Tsol_L$ provided by Proposition \ref{TLinversion}. Indeed,
\begin{align*} \widehat{g}(\chi) = 
\begin{cases}
\widehat{f}(\gamma) & \text{if $\chi = \varphi(\gamma)$ for some $\gamma \in R$},\\
0                   & \text{otherwise.}
\end{cases}
\end{align*}
Hence $\int_H g = \ghat(1_{\Hhat}) = \fhat(1_{\Ghat}) = \int_G f$. Now let $L = c_1 x_1 + \cdots + c_t x_t$ be a $(G,H,k)$-admissible form of multiplier-height at most $h$; thus we have integers $n_1, \ldots, n_t$ such that $n_1 c_1 + \cdots + n_t c_t = 1$ and $\abs{n_1} + \cdots + \abs{n_t} \leq h$. Proposition \ref{TLinversion} gives that
\[ \Tsol_L(f) = \sum_{\substack{\gamma \in \Ghat \\ \gamma^{c_i} \in R\,\, \forall i}} \fhat(\gamma^{c_1}) \cdots \fhat(\gamma^{c_t}) \quad \text{ and }\quad \Tsol_L(g) = \sum_{\substack{\chi \in \Hhat \\ \chi^{c_i} \in \varphi(R)\,\, \forall i}} \ghat(\chi^{c_1}) \cdots \ghat(\chi^{c_t}). \]
Let us write $\Gamma := \{ \gamma \in \Ghat : \gamma^{c_i} \in R \text{ for all $i$} \}$ and $\Psi := \{ \chi \in \Hhat : \chi^{c_i} \in \varphi(R) \text{ for all $i$} \}$ for the index sets occurring in these two sums. The result will follow if we can show that $\varphi$ is a bijection from $\Gamma$ to $\Psi$ such that $\varphi(\gamma)^{c_i} = \varphi(\gamma^{c_i})$ for all $i$ and $\gamma \in \Gamma$. We certainly have the second property, for if $\gamma \in \Gamma$ then
\[ \gamma = \gamma^{n_1 c_1 + \cdots + n_t c_t} = (\gamma^{c_1})^{n_1} \cdots (\gamma^{c_t})^{n_t} \in R^{n_1} \cdots R^{n_t} \subset Q, \]
and $\varphi$ is a $k$-isomorphism on $Q$, where $k \geq \abs{c_i}$ for all $i$. So we just need to establish that $\varphi(\Gamma) = \Psi$. Let us first deal with $\varphi(\Gamma) \subset \Psi$: let $\gamma \in \Gamma$. Then we need to show that $\varphi(\gamma)^{c_i} \in \varphi(R)$ for all $i$. But this is immediate since $\varphi(\gamma)^{c_i} = \varphi(\gamma^{c_i})$. The opposite inclusion follows in the same way using $\varphi^{-1}$ since $\Psi \subset \varphi(Q)$, which follows from the fact that if $\chi \in \Psi$ then
\[ \chi = (\chi^{c_1})^{n_1} \cdots (\chi^{c_t})^{n_t} \in \varphi(R)^{n_1} \cdots \varphi(R)^{n_t} = \varphi(R^{n_1} \cdots R^{n_t}) \subset \varphi(Q), \]
$\varphi$ being a Freiman $h$-isomorphism.
\end{proof}

The rest of the proof of Proposition \ref{transfer} is identical to that of the proof of Proposition 4.2.8 in \cite{OlofTh}, and so we are done.
\end{proof}

Proposition \ref{transfer} gives us a criterion for transferring functions between two compact abelian groups. The following lemmas show that this criterion allows us to work with the groups $\Zmod{p}$ and $\T$, since $\widehat{\Zmod{p}} \cong \Zmod{p}$ and $\widehat{\T} \cong \Z$.

\begin{lemma}
Let $n, k \in \N$ and let $p \geq (2k)^n$ be a prime. Then for any set $A \subseteq \Zmod{p}$ of size $n$ there is a set $B \subset \Z$ that is Freiman $k$-isomorphic to $A$.
In other words, $\Z$ can Freiman $(n,k)$-model $\Zmod{p}$ provided $p \geq (2k)^n$.
\end{lemma}

This result is standard and follows from an application of Dirichlet's box principle. The proof of the following lemma from \cite{OlofTh} is slightly more subtle but still elementary.

\begin{lemma}
Let $n, k$ be positive integers and let $N \geq (4k)^n$ be an integer. Then for any set $A \subseteq \Z$ of size $n$ there is a set $B \subset \Zmod{N}$ that is Freiman $k$-isomorphic to $A$.
In other words, $\Zmod{N}$ can Freiman $(n,k)$-model $\Z$ provided $N \geq (4k)^n$.
\end{lemma}

Thus we obtain the following immediate corollary of Proposition \ref{transfer}.

\begin{corollary}\label{transfer-cor}
Let $\epsilon>0$ and $h$ be a positive integer, and let each of $G$ and $H$ be $\Zmod{p}$ or $\T$, where $p \geq C(\epsilon, h)$ is a prime. Then for any measurable function $f : G \to [0,1]$ with $\int_G f = \alpha$ there is a continuous function $g : H \to [0,1]$ with $\int_H g = \alpha$ such that
\[ \abs{\Tsol_L(f) - \Tsol_L(g)} \leq t \epsilon \alpha^{t-2} \]
for any $(G,H,k)$-admissible linear form $L$ in $t \geq 3$ variables with multiplier-height at most $h$.
\end{corollary}

\section{From functions to sets}\label{functions-to-sets}
In our proof of Theorem \ref{mainres} we shall use Proposition \ref{transfer} to obtain a function $g$ with certain properties, and we shall then require a set with similar properties. The existence of such a set will be guaranteed by the following result.

\begin{lemma}\label{MainF2S}
Let $\epsilon > 0$, let $G$ be $\T$ or $\Zmod{p}$ for $p$ sufficiently large, and let $f : G \to [0,1]$ be measurable. Then there exists a measurable set $A\subset G$ such that $\left|\mu_G(A)-\int_G f\right|\leq \epsilon$ and
 $|\Tsol_L(A)-\Tsol_L(f)|\leq t\epsilon$ for any $G$-admissible linear form $L$ in $t \geq 3$ variables.
\end{lemma}

Note that any linear form is $\T$-admissible and most forms are $\Zmod{p}$-admissible. To prove this lemma we shall use Theorem \ref{U2control} and a probabilistic construction that is familiar in the setting of finite abelian groups. 

Given two sets $A,B$ let $A\Delta B$ denote their symmetric difference. The following notion of discretization will be used here and in the next section.

\begin{defn}
A set $A \subset \T$ is \emph{$(\delta,n)$-measurable} if there exists a set $B$ which is the union of dyadic intervals $I_{n,j} := [(j-1)/2^n, j/2^n)$, $j\in [2^n]$, such that $\mu(A\Delta B)<\delta$.

A function $f : \T \to [0,1]$ is \emph{$(\delta,n)$-measurable} if there is a function $g : \T \to [0,1]$ that is constant on the intervals $I_{n,j}$ such that $\norm{f-g}_{\mathcal{L}_1} < \delta$.
\end{defn}

We shall use both of the following equivalent instances of Littlewood's first principle, which are discussed in \cite[\S 2.4]{tao-year1}.

\begin{lemma}\label{LW}
Let $A \subset \T$ be measurable. Then for every $\delta > 0$ there exists $n$ such that $A$ is $(\delta, n)$-measurable.
\end{lemma}

\begin{lemma}\label{LW-fns}
Let $f : \T \to [0,1]$ be measurable. Then for every $\delta > 0$ there exists $n$ such that $f$ is $(\delta, n)$-measurable.
\end{lemma}
While one cannot in general approximate a $[0,1]$-valued function by a set in $\mathcal{L}_1(G)$, one can do so in $U^2(G)$ in the sense
of the following result. It is this that allows us to establish Lemma 4.1.
\begin{lemma}\label{F2S}
Let $\epsilon > 0$, let $G$ be $\T$ or $\Zmod{p}$ for $p$ sufficiently large, and let $f : G \to [0,1]$ be measurable.
Then there is a measurable set $A \subset G$ such that $\lVert f - 1_A \rVert_{U^2} \leq \epsilon$. Moreover, for $G = \T$ we may take $A$ to be a finite union of dyadic intervals.
\end{lemma}
\begin{proof}
For $G=\Zmod{p}$ this a standard result: define a subset $A$ of $G$ randomly by letting $x \in A$ with probability $f(x)$ independently for each $x \in G$. Then by independence the expectation of $\norm{f-1_A}_{U^2}^4$ equals 0, up to an error of magnitude $O(p^{-1})$ due to averaging over degenerate parallelograms (i.e. parallelograms $(x+y,x+y',x'+y,x'+y')$ with at least two vertices being equal). Thus there exists a choice of $A$ for which $\norm{f-1_A}_{U^2}=O(p^{-1/4})$.

Now we adapt this standard argument to deal with the group $\T$. First we approximate $f$ by a suitable step-function: by Lemma \ref{LW-fns} there is an integer $n$ and coefficients $\gamma_j\in [0,1]$ such that the function $g=\sum_{j\in [2^n]} \gamma_j 1_{I_{n,j}}$ satisfies $\norm{f-g}_{\mathcal{L}_1} < \epsilon^4$. We thus have $\norm{f-g}_{U^2} \leq \norm{f-g}_{\mathcal{L}_1}^{1/4}<\epsilon$.

Next we define an appropriate finite probability space. Consider the set $\Omega$ of functions of the form $1_A(x)=\sum_{j\in [2^n]} \alpha_j 1_{I_{n,j}}(x)$, where for all $j$ we have $\alpha_j\in\{0,1\}$. An element
in $\Omega$ can be identified with an element $\alpha=(\alpha_1,...,\alpha_{2^n})$ of $\{0,1\}^{2^n}$, and we can then
define a probability on $\Omega$ by declaring the events $\{\alpha_j=1\}$ to be independent and assigning to $\{\alpha_j=1\}$
the probability $\gamma_j$.

Now we compute the expectation of $\norm{g-1_A}_{U^2}^4$ for a randomly chosen $1_A \in \Omega$, using the following familiar expression for the $U^2$ norm:
\[\norm{f}_{U^2(G)}^4=\int_{G^3} f(x)\,\overline{f(x+h)}\,\overline{f(x+k)}\,f(x+h+k) \ud x \ud h \ud k.\]
This expression is seen to equal \eqref{U2} by Lemma \ref{cts-surj-hom}.
By Fubini's theorem we have
\[
\mathbb{E} \norm{g-1_A}_{U^2}^4=\int_{\T^3}
\mathbb{E} (g-1_A)(x)(g-1_A)(x+h)(g-1_A)(x+k)(g-1_A)(x+h+k)\; \ud x\ud h\ud k.
\]
Fix any values for $x,h,k$ such that no two vertices of the corresponding parallelogram $(x,x+h,x+k,x+h+k)$ lie in the same interval
$I_{n,j}$. Then the expectation of the product $(g-1_A)(x)(g-1_A)(x+h)(g-1_A)(x+k)(g-1_A)(x+h+k)$ is 0, by independence. Now consider the integral of this product over values of $x,h,k$ such that at least two vertices of the corresponding parallelogram lie in the same interval. This integral is at most the Haar measure of the set $D$ of these parallelograms,
\[
D=\{(x,h,k)\in \T^3: \exists \omega,\omega'\in\{0,1\}^2, \omega\neq\omega', \exists
j\in [2^n], I_{n,j}\ni x+\omega\cdot (h,k),x+\omega'\cdot (h,k)\},
\]
where $\omega\cdot (h,k)=\omega_1 h+\omega_2 k$. Note that $D$ is the support of the measurable function
\[
F_D:(x,h,k)\mapsto \sum_{\substack{\omega,\omega'\in \{0,1\}^2 \\\omega\neq\omega'}} \sum_{j\in[2^n]}
1_{I_{n,j}}\big(x+\omega\cdot (h,k)\big) 1_{I_{n,j}}\big(x+\omega'\cdot (h,k)\big),
\]
so $D$ is a measurable subset of $\T^3$. We want to show that the Haar measure of $D$ vanishes as $n\rightarrow \infty$, and for that it clearly suffices to show that the integral of $F_D$ vanishes. Fix any pair of distinct $\omega,\omega'$. Then $\omega-\omega'$ has at least one non-zero coordinate (where the subtraction here is coordinate-wise). For a fixed $x$, we have $x+\omega\cdot (h,k),x+\omega'\cdot (h,k)$ in the same interval $I_{n,j}$ only if $(\omega-\omega')\cdot (h,k)\in (-2^{-n},2^{-n})$. It follows that, for this pair $\omega,\omega'$, the integral
\[
\int_{\T^3} \sum_{j\in[2^n]}
1_{I_{n,j}}\big(x+\omega\cdot (h,k)\big) 1_{I_{n,j}}\big(x+\omega'\cdot (h,k)\big)\ud x\ud h \ud k
\]
is at most the Haar measure of the slab $S=\{(h,k)\in\T^2:(\omega-\omega')\cdot (h,k)\in (-2^{-n},2^{-n})\}$, which is $2^{-n+1}$. Applying this argument to each pair $\omega,\omega'$ shows that the Haar measure of $D$ vanishes as required.
\end{proof}

\begin{remark}
A minor modification of the above proof allows one to replace the $U^2$ norm by the $U^d$ norm for any $d \geq 2$; these norms are useful generalizations of the $U^2$ norm \cite{T-V}.
\end{remark}

\begin{proof}[Proof of Lemma \ref{MainF2S}]
Let $G$ be either $\T$ or $\Zmod{p}$, where $p$ is large. For a measurable function $f : G\to[0,1]$, let $A$ be the set given by Lemma \ref{F2S} applied with initial parameter $\epsilon$. Then by the Cauchy-Schwarz inequality we have $\abs{\int_G 1_A-\int_G f} \leq \norm{1_A-f}_{U^2}\leq\epsilon$. Now let $L$ be any $G$-admissible linear form in at least three variables. Then multilinearity of $S_L$ and Theorem \ref{U2control} imply $\abs{\Tsol_L(f)-\Tsol_L(A)} \leq t \norm{f-1_A}_{U^2}\leq t \epsilon$.
\end{proof}

\section{Removal lemmas}\label{Removal}

In this section, for $G=\T$ or $\Zmod{p}$, we are interested in measurable subsets $A\subset G$ of positive measure such that $\Tsol_L(A)<\delta$ for some small $\delta>0$. We shall prove the following result.

\begin{lemma}\label{Tremoval-t}
Let $\epsilon>0$ and let $L$ be a linear form in $t$ variables. There exists $\delta>0$ such that, for any measurable sets $A_1,A_2,\ldots,A_t\subset\T$ satisfying $\Tsol_L(A_1,A_2,\ldots,A_t)<\delta$, there are measurable sets $E_i\subset\T$ with $\mu(E_i)< \epsilon$ such that $(A_1\setminus E_1) \times (A_2\setminus E_2)\times\cdots\times (A_t\setminus E_t) \cap \ker L =\emptyset$.
\end{lemma}

This result has a well-known analogue for finite abelian groups, an analogue that now has several proofs. One such proof, given in \cite{KSVRL}, proceeds by turning the problem into one of removing small subgraphs from a certain auxiliary graph, and then applying a removal lemma for graphs. Unfortunately this argument does not seem to extend straightforwardly to the infinitary setting of the group $\T$. An earlier proof was given by Green \cite{GAR}, who established a Fourier-analytic regularity lemma from which the finitary removal lemma follows. This proof does generalize to the infinitary setting, allowing one to establish a removal lemma for arbitrary compact abelian groups, but checking this is somewhat technical. Instead, we shall prove Lemma \ref{Tremoval-t} using a finitary removal lemma as a black box. We first reduce to the special case where $L$ is the linear form $L(x)=x_1+\cdots+x_t$; this form will be denoted by $\mathbf{1}$ throughout this section.

\begin{lemma}\label{offdiag}
For any $\epsilon>0$ and any positive integer $t$ there exists $\delta>0$ such that the following holds. Let $A_1,A_2,\ldots,A_t$ be measurable subsets of $\T$ such that $\Tsol_{\mathbf{1}}(A_1,\ldots,A_t)<\delta$. Then there exist measurable sets $E_i\subset\T$ such that $(A_1\setminus E_1)\times\cdots\times(A_t\setminus E_t)\cap\ker\mathbf{1}=\emptyset$ and $\mu(E_i)<\epsilon$ for each $i$.
\end{lemma}

When $N$ is prime, the equivalence between the $\Zmod{N}$-analogues of Lemmas \ref{Tremoval-t} and \ref{offdiag} follows simply from inverting the dilations. For the circle we can still prove the desired equivalence at the cost of a slight worsening in the dependence between the parameters.
To show that Lemma \ref{offdiag} implies Lemma \ref{Tremoval-t} we use the following.
\begin{lemma}\label{T_L_L'}
Let $L(x) = c_1 x_1 + \cdots + c_t x_t$ and let $A_1,A_2,\ldots,A_t\subset \T$ be measurable. Then
\[\Tsol_L(A_1,\ldots,A_t) \leq \Tsol_\mathbf{1}(c_1 A_1, \ldots, c_t A_t) \leq |c_1 \cdots c_t|\, \Tsol_L(A_1,\ldots,A_t).\]
\end{lemma}
\begin{proof}
The map $\phi:(x_1,\ldots,x_t)\to (c_1x_1,\ldots,c_tx_t)$ is a continuous endomorphism on $\T^t$ which restricts to a continuous surjective homomorphism from $\ker L$ to $\ker \mathbf{1}$. In fact, since $L$ is the composition $\mathbf{1} \circ \phi$,
we have $\ker L = \phi^{-1} \ker \mathbf{1}$. By Lemma \ref{cts-surj-hom} we therefore have $\mu_\mathbf{1}=\mu_L\circ \phi^{-1}$. It follows that
$\Tsol_\mathbf{1}(A_1,\ldots,A_t)=\Tsol_L(c_1^{-1} A_1,\ldots,c_t^{-1} A_t)$, where $c_i^{-1}$ denotes taking the preimage under dilation by $c_i$.

Applying this to the sets $c_i A_i$ gives $\Tsol_\mathbf{1}(c_1 A_1,\ldots, c_t A_t)=\Tsol_L(c_1^{-1} c_1 A_1,\ldots,c_t^{-1} c_t A_t)$,
and it is easy to check that $c_i^{-1} c_i A_i = \bigcup_{j \in \Zmod{c_i}} \left( A_i + j/c_i \right)$.
The claim then follows from basic properties of $\Tsol_L$, namely multilinearity and invariance under translation by elements of $\ker L$.
\end{proof}

\begin{proof}[Proof of Lemma \ref{Tremoval-t} using Lemma \ref{offdiag}]
Fix $\epsilon>0$ and let $\delta>0$ be as given by Lemma \ref{offdiag}. Now suppose $A_1,A_2,\ldots,A_t\subset \T$ are measurable subsets satisfying $\Tsol_L(A_1,\ldots,A_t)<\delta/|c_1\cdots c_t|$. By Lemma \ref{T_L_L'} we then have
$\Tsol_{\mathbf{1}}(c_1A_1,\ldots,c_tA_t) < \delta$. Lemma \ref{offdiag} therefore provides us with sets $F_i$ with $\mu(F_i) < \epsilon$ such that
$(c_1A_1\setminus F_1)\times\cdots\times(c_tA_t\setminus F_t) \cap \ker\mathbf{1}=\emptyset$.
Let $E_i = c_i^{-1} F_i$, and let $\phi$ be the map from Lemma \ref{T_L_L'}. Since
\[ \phi\big((A_1\setminus E_1)\times\cdots\times(A_t\setminus E_t) \cap\ker L \big) = (c_1A_1\setminus F_1)\times\cdots\times(c_tA_t\setminus F_t)\cap\ker\mathbf{1}=\emptyset,\]
the sets $E_i$ satisfy the conclusion of Lemma \ref{Tremoval-t}.
\end{proof}

We shall now prove Lemma \ref{offdiag}. For notational convenience we assume in the proof that $\mathbf{1}$ has last coefficient $c_t=-1$, without loss of generality. Note that as a special case of \eqref{param} we have
\begin{equation}\label{T_1}
\Tsol_{\mathbf{1}}(A_1,\ldots, A_t)= \int_{\T^{t-1}} 1_{A_1}(x_1)\cdots 1_{A_{t-1}}(x_{t-1}) 1_{A_t}(x_1+x_2+\cdots+x_{t-1}) \ud x_1\cdots \ud x_{t-1}.
\end{equation}

Using Lemma \ref{LW} one may reduce the proof of Lemma \ref{offdiag} to establishing the following variant.

\begin{lemma}\label{removal_lemma_discrete}
For any $\epsilon > 0$ and $t\in \N$ there exists $\delta > 0$ such that the following holds. Let $A_1, A_2,\ldots, A_t \subset \T$ be unions of intervals $[(j-1)/N, j/N)$, $j\in [N]$, for some $N \in \N$. Suppose further that $\Tsol_\mathbf{1}(A_1,\ldots, A_t)< \delta$.
Then there are measurable sets $E_i \subset A_i$ with $\mu(E_i)<\epsilon$ such that
$(A_1 \setminus E_1)\times\cdots\times(A_t \setminus E_t)\,\cap\ker\mathbf{1} = \emptyset$.
\end{lemma}

\begin{proof}[Proof of Lemma \ref{offdiag} using Lemma \ref{removal_lemma_discrete}]
Fix $\epsilon>0$ and let $\delta \in (0, \epsilon)$ be such that Lemma \ref{removal_lemma_discrete} holds with initial parameter $\epsilon/2$. Let
$A_1,\ldots,A_t\subset\T$ be measurable sets satisfying $\Tsol_\mathbf{1}(A_1,\ldots,A_t)< \delta/(t+1)$. Using \eqref{T_1} and a telescoping expansion (multilinearity of $\Tsol_\mathbf{1}$) we have $|\Tsol_\mathbf{1}(A_1,\ldots,A_t)-\Tsol_\mathbf{1}(B_1,\ldots,B_t)|\leq \mu(A_1\Delta B_1)+\cdots+\mu(A_t\Delta B_t)$
for any measurable sets $B_1,\ldots,B_t$. From Lemma \ref{LW} we obtain a positive integer $n$ and sets $B_i$ that are unions of intervals $[(j-1)2^{-n},j2^{-n})$ such that $\mu(A_i\Delta B_i)\leq \delta/(t+1)$ for each $i \in [t]$. It follows that $\Tsol_\mathbf{1}(B_1,\ldots,B_t)< \delta$. We apply Lemma \ref{removal_lemma_discrete} to these discretized sets, obtaining measurable sets $F_1,\ldots,F_t$ such that $\mu(F_i)< \epsilon/2$ and
\[ (B_1 \setminus F_1)\times\cdots\times(B_t \setminus F_t) \cap \ker\mathbf{1} = \emptyset. \]
Now letting $E_i=F_i\cup (A_i\setminus B_i)$, we have $\mu(E_i)<\epsilon$, and $A_i \setminus E_i\subset B_i\setminus F_i$, whence
$(A_1 \setminus E_1)\times\cdots\times(A_t \setminus E_t) \cap \ker\mathbf{1} = \emptyset$ as required.
\end{proof}

We shall deduce Lemma \ref{removal_lemma_discrete} from the following analogue for $\Zmod{N}$, proved in \cite{GAR}.

\begin{lemma}\label{removal_lemma_finite}
For any $\epsilon > 0$ and $t\in\N$ there exists $\delta > 0$ such that the following holds. Let $N \in \N$ and suppose that $A_1,\ldots, A_t \subset \Zmod{N}$ and that $\Tsol_\mathbf{1}(A_1,\ldots, A_t)< \delta$. 
Then there are sets $E_i \subset A_i$ with $\abs{E_i} \leq \epsilon N$ such that
$(A_1 \setminus E_1)\times\cdots\times(A_t \setminus E_t) \cap \ker\mathbf{1}=\emptyset$.
\end{lemma}
To use this we need to express the solution measure $\Tsol_\mathbf{1}$ on $\T$ in terms of solution measures on $\Z_N$.
Recall that the Eulerian number $\eulerian{n}{k}$ is the number of permutations $a_1, \ldots, a_n$ of $[n]$ in which there are precisely $k$ values of $i$ such that $a_i < a_{i+1}$.

\begin{lemma}\label{discrete_solutions}
Let $t\geq 3$ and suppose $A_1,\ldots, A_t \subset \T$ are unions of intervals $[(j-1)/N, j/N)$, $j\in [N]$, for some positive integer $N$. Let $A_i'\subset \Zmod{N}$ be defined by $1_{A_i'}(x)=1_{A_i}(x/N)$. Then 
\[ \Tsol_\mathbf{1}(A_1,\ldots, A_t) = \frac{1}{(t-1)!}\sum_{r=0}^{t-2} \eulerian{t-1}{r}\, \Tsol_\mathbf{1}(A_1',\ldots, A_t'-r). \]
In particular, $\Tsol_\mathbf{1}(A_1',\ldots, A_t'-r) \leq (t-1)!\, \Tsol_\mathbf{1}(A_1,\ldots, A_t)$ for any $r\in \{0,1, \ldots, t-2\}$.
\end{lemma}
\begin{proof}
For any $x \in \T$ we have $1_{A_i}(x) = 1_{A_i}( \lfloor N x \rfloor / N )$, where multiplication by $N$, floor function, and division by $N$ are all taken in $\R$.
Then $\Tsol_\mathbf{1}(A_1,\ldots, A_t)$ equals
\begin{eqnarray*}
&& \int_{\T^{t-1}} 1_{A_1}(\lfloor Nx_1 \rfloor /N)\cdots 1_{A_t}( \lfloor N(x_1+\cdots+x_{t-1}) \rfloor/N) \ud x_1\cdots \ud x_{t-1} \\
&=& \sum_{a_1,\ldots,a_{t-1}\in \Zmod{N}} 1_{A_1}(a_1/N)\cdots 1_{A_{t-1}}(a_{t-1}/N) \\
&&\hspace{1.5cm}\cdot\int_{[0,1/N)^{t-1}} 1_{A_t}( \lfloor N((a_1/N + y_1)+\cdots+ (a_t/N + y_{t-1})) \rfloor/N) \ud y_1\cdots \ud y_{t-1} \\
&=& \frac{1}{N^{t-1}} \sum_{a_1,\ldots,a_{t-1} \in \Zmod{N}} 1_{A_1'}(a_1)\cdots 1_{A_{t-1}'}(a_{t-1})\\
&&\hspace{1.5cm}\cdot\int_{[0,1)^{t-1}} 1_{A_t'}(a_1+\cdots+a_{t-1}+\lfloor y_1+\cdots+ y_{t-1} \rfloor) \ud y_1\cdots \ud y_{t-1}.
\end{eqnarray*}
For any real-valued function $f$ on the integers, \cite[6.65]{graham-knuth-patashnik} gives
\[ \int_0^1 \cdots \int_0^1 f(\lfloor x_1 + \cdots + x_n \rfloor) \ud x_1 \cdots \ud x_n = \sum_{k=0}^{n-1} \eulerian{n}{k} \frac{f(k)}{n!}. \]
The result now follows immediately.
\end{proof}

\begin{proof}[Proof of Lemma \ref{removal_lemma_discrete}]
We can assume that $t\geq 3$. Given $\epsilon > 0$, set $\epsilon' = \epsilon/(t-1)$ and let $\delta'$ be given by Lemma \ref{removal_lemma_finite} applied with initial parameter $\epsilon'$. Define $\delta = \delta'/(t-1)!$ and let $A_1,\ldots,A_t \subset \T$ be sets that are unions of intervals of the form $[j/N, (j+1)/N)$ for some positive integer $N$ and satisfy $\Tsol_\mathbf{1}(A_1,\ldots, A_t)< \delta$.

Let the sets $A_i' \subset \Zmod{N}$ be as in Lemma \ref{discrete_solutions}. We then have $\Tsol_\mathbf{1}(A_1',\ldots, A_t'-r)< \delta'$ for each $r = 0, \ldots, t-2$, and so Lemma \ref{removal_lemma_finite} gives us sets $E_{i,r}' \subset A_i'$ such that $|E_{i,r}'| \leq \epsilon' N$ and
\begin{equation}\label{removed'}
(A_1' \setminus E_{1,r}')\times\cdots\times\big((A_t'-r) \setminus (E_{t,r}'-r)\big) \cap \ker\mathbf{1}=\emptyset.
\end{equation}
Setting $E_i' = \bigcup_r E_{i,r}'$ we therefore see that there are no solutions to $y_1 + \cdots + y_{t-1} = y_t - r$ with $y_i \in A_i' \setminus E_i'$ for $r \in [0, t-2] \subset \Zmod{N}$.

We now define corresponding removal-sets in $\T$: for each $i\in [t]$ let
\[ E_i = \bigcup_{x \in E_i'} [x/N, (x+1)/N). \]
Since $\abs{E_i'} \leq \epsilon N$ we have $\mu(E_i) \leq \epsilon$ for each $i$. Now suppose that we had an element
\[ (x_1, \ldots, x_t)\; \in \;(A_1 \setminus E_1) \times \cdots \times (A_t \setminus E_t) \cap \ker \mathbf{1}. \]
Then, since each of the sets $A_i \setminus E_i$ is a union of intervals $[j/N, (j+1)/N)$, we have that the element $\lfloor N x_i \rfloor$ of $\Zmod{N}$ lies in $A_i' \setminus E_i'$. But we also have $x_t = x_1 + \cdots + x_{t-1}$, whence $\lfloor N x_t \rfloor = \lfloor N x_1 \rfloor + \cdots + \lfloor N x_{t-1} \rfloor + r$ for some $r \in [0, t-2] \subset \Zmod{N}$. This contradicts \eqref{removed'}, and therefore 
$(A_1 \setminus E_1) \times \cdots \times (A_t \setminus E_t) \cap \ker \mathbf{1} = \emptyset$ as required.
\end{proof}

\section{Proofs of convergence and positivity}\label{Final}

We can now prove our main result.

\begin{proof}[Proof of Theorem \ref{mainres}]
We are given a finite family $\mathcal{F}$ of linear forms. Let $n$ be the cardinality of $\mathcal{F}$, and fix $\epsilon>0$.
Let $\delta \in (0, \epsilon/2)$ be such that Lemma \ref{Tremoval-t} and its analogue for $\Zmod{p}$ both work for each $L\in\mathcal{F}$, with initial parameter $\epsilon/2n$. Let $t$ be the maximum number of variables that occurs in a form in $\mathcal{F}$ and let $C=C(\epsilon,\mathcal{F})$ be such that Corollary \ref{transfer-cor} (transference) and Lemma \ref{MainF2S} (functions to sets) both work when applied with initial parameter $\delta/2t$ for any $p > C$ and $L\in\mathcal{F}$. We claim that $\abs{ \md_\mathcal{F}(\T) - \md_\mathcal{F}(\Zmod{p}) } \leq \epsilon$ for any prime $p > C$.

To see that $\md_\mathcal{F}(\T)\geq \md_\mathcal{F}(\Zmod{p})-\epsilon$, let $\alpha= \md_\mathcal{F}(\Zmod{p})$ and let $A\subset \Zmod{p}$ be $\mathcal{F}$-free with $\mu(A)=\alpha$. Corollary \ref{transfer-cor} then gives us a measurable function $g : \T\to [0,1]$ with $\int_\T g=\alpha$ such that $\Tsol_L(g) < \delta/2$ for every $L\in\mathcal{F}$. Applying Lemma \ref{MainF2S} to $g$ with initial parameter $\delta/2t$, we obtain a measurable subset $B$ of $\T$ of density at least $\alpha-\delta$ such that $\Tsol_L(B) < \delta$ for every $L\in\mathcal{F}$.

We now apply the removal lemma on $\T$. By our choice of $\delta$, Lemma \ref{Tremoval-t} gives us an $\mathcal{F}$-free subset $D$ of $B$ with $\mu(D)\geq \mu(B)- n \epsilon/2n\geq \alpha -\delta-\epsilon/2$. Therefore $\md_\mathcal{F}(\T)\geq \md_\mathcal{F}(\Zmod{p})-\epsilon$ as required.

The same argument, but with the roles of $\Zmod{p}$ and $\T$ swapped, shows that we also have $\md_\mathcal{F}(\Zmod{p})\geq \md_\mathcal{F}(\T)-\epsilon$, and so the result follows.
\end{proof}
\begin{remark}\label{rem}
One of our aims for the argument above was to treat the direction $d_\mathcal{F}(\Zmod{p})>d_\mathcal{F}(\T)-\epsilon$ and its opposite in a unified manner. It should be noted however that the first direction can also be treated more directly, without using Fourier analysis. In a nutshell, if $A\subset \T$ is $\mathcal{F}$-free with $\mu(A)>d_\mathcal{F}(\T)-\epsilon$ then the continuity of the forms in $\mathcal{F}$ implies that one can find an $\mathcal{F}$-free \emph{open} set $A'$ of measure at least $d_\mathcal{F}(\T)-2\epsilon$, and then for large $p$ the set $A_p=\{x\in \Zmod{p}: x/p\in A'\}$ is $\mathcal{F}$-free and of density at least $d_\mathcal{F}(\T)-3\epsilon$ in $\Zmod{p}$.
\end{remark}
We now prove Proposition \ref{positivity}, which said that $\md_\mathcal{F}(\T)$ is positive for any finite family $\mathcal{F}$ of non-invariant forms. We can do this by modifying an idea employed by Ruzsa \cite{ruzsa:linear-equationsII}. For a form $L(x) = c_1 x_1 + \cdots + c_t x_t$ we write $s_L=\abs{c_1} + \cdots + \abs{c_t}$.

\begin{proposition}
Let $\mathcal{F}$ be a finite family of non-invariant linear forms. Then
\[ \md_\mathcal{F}(\T) \geq \left(\sum_{L \in \mathcal{F}} s_L\right)^{-1}. \]
\end{proposition}
\begin{proof}
Let $s = \sum_{L \in \mathcal{F}} s_L$ and let $A$ be the interval $(-1/2s, 1/2s)$ embedded in $\T$. We claim that there is a translate $A-y$ of $A$ that is $\mathcal{F}$-free. This will be the case if $L(y,\ldots,y) \notin L(A \times \cdots \times A) = (-s_L/2s, s_L/2s)$ for each $L \in \mathcal{F}$. Since each $L$ is non-invariant, each such condition on $y$ excludes a finite union of open intervals, with total length $s_L/s$. Thus all the conditions together exclude a finite union of open intervals, the lengths of which sum to at most $1$. Hence there is some $y \in \T$ outside this union of intervals, and the result follows.
\end{proof}

\section{Concluding remarks}

Our proof of Theorem \ref{mainres} consists essentially of a combination of the transference result Proposition \ref{transfer} with the removal result Lemma \ref{Tremoval-t}. Two remarks should be made about this. The first is that the combination of Fourier-regularization (a key tool in the proof of the transference result) and removal results has been used successfully before; in particular, Green \cite[Theorem 9.3]{GAR} used it to relate the number of subsets of $[N]$ that are free from non-trivial solutions to $L(x)=0$ to the maximum size of an $L$-free subset of $[N]$. The second remark is that the argument in our proof of Theorem \ref{mainres} readily yields analogues of the theorem for many other families of groups, provided in particular that the appropriate removal lemmas are available (examples of such families include $(\Zmod{kp})$, where $k\in \N$ is fixed and $p$ ranges over the primes). We have not treated such generalizations here in order to avoid certain technicalities. Let us note, however, that if all one is interested in is convergence of maximal densities, rather than convergence to a particular quantity on a group, then all the theory one needs is finitary and the appropriate removal lemmas are well-known.

It would be interesting to generalize Theorem \ref{mainres} to allow the family $\mathcal{F}$ to consist not just of single linear equations but also of systems of linear equations. To this end it can be useful to classify systems according to a notion of \emph{complexity} related to the Gowers norms (see \cite{gowers-wolf}). The methods of this paper readily extend to give convergence of $d_\mathcal{F}(\Z_p)$ for a family of systems of complexity 1, but establishing $\T$ as a limit group along these lines requires an extension of Lemma \ref{Tremoval-t}. Convergence for systems of greater complexity requires other methods.

Finally, regarding the original question of Ruzsa mentioned in the introduction, we note that there is a simple transference result for functions on $[N]$, proved using an argument somewhat different from that employed for transference here, and that this can be used to answer Ruzsa's question affirmatively (we shall detail this elsewhere).\\

\textbf{Acknowledgements.} The authors are grateful to Tom Sanders for helpful comments and to an anonymous referee for pointing out the argument in Remark \ref{rem}. The second-named author would also like to thank Ben Green for his supervision and encouragement that led to the main results of \cite[Chapter 4]{OlofTh}.

\end{document}